\documentclass[12pt]{article}%
\usepackage{enumerate}
\usepackage{graphicx,psfrag}
\usepackage{graphicx}
\usepackage{amsmath}
\usepackage{amsfonts}
\usepackage{amssymb}%
\providecommand{\U}[1]{\protect\rule{.1in}{.1in}}

\newtheorem{theorem}{Theorem}

\newtheorem{lemma}[theorem]{Lemma}

\newtheorem{proposition}[theorem]{Proposition}

\newenvironment{proof}[1][Proof]{\textbf{#1.} }{\ \rule{0.5em}{0.5em}}
\begin{document}

\title{Maximizing the number of edges in optimal $k$-rankings}
\author{Rigoberto Fl\'{o}rez \thanks{
A part of this work was performed at The Citadel and supported by The Citadel Foundation, and a
second part of the work was performed at University of South Carolina Sumter,
Corresponding address:\ Department of
Mathematics and Computer Science, The Citadel, Charleston, SC 29409,
Email:\ rigo.florez@citadel.edu}\\The Citadel\\Charleston, SC 29409
\and Darren A. Narayan\thanks{Corresponding address:\ School of Mathematical
Sciences, Rochester Institute of Technology, Rochester, NY 14623-5604,
Email:\ darren.narayan@rit.edu}\\Rochester Institute of Technology\\Rochester, NY 14623-5604}
\maketitle

\begin{abstract}
A $k$-ranking is a vertex $k$-coloring such that if two vertices have the same
color any path connecting them contains a vertex of larger color. The
\textit{rank number} of a graph is smallest $k$ such that $G$ has a
$k$-ranking. For certain graphs $G$ we consider the maximum number of edges
that may be added to $G$ without changing the rank number. Here we investigate
the problem for $G=P_{2^{k-1}}$, $C_{2^{k}}$, $K_{m_{1},m_{2},\dots,m_{t}}$,
and the union of two copies of $K_{n}$ joined by a single edge. In addition to
determining the maximum number of edges that may be added to $G$ without
changing the rank number we provide an explicit characterization of which
edges change the rank number when added to $G$, and which edges do not.

\end{abstract}

\section{Introduction}

A \emph{vertex coloring} of a graph is a labeling of the vertices so that no
two adjacent vertices receive the same label. A $k$-\emph{ranking} of a graph
is a coloring of the vertex set with $k$ positive integers such that on every
path connecting two vertices of the same color there is a vertex of larger
color. The \textit{rank number of a graph} is defined
to be the smallest $k$ such that $G$ has a $k$-ranking.

Early studies involving the rank number of a graph were sparked by its
numerous applications including designs for very large scale integration
layout (VLSI), Cholesky factorizations of matrices, and the scheduling of
manufacturing systems \cite{torre,leiserson,sen}. Bodlaender et al. proved
that given a bipartite graph $G$ and a positive integer $n$, deciding whether
a rank number of $G$ is less than $n$ is NP-complete \cite{bodlaender}. The
rank number of paths, cycles, split graphs, complete multipartite graphs,
powers of paths and cycles, and some grid graphs are well known
\cite{alpert,bodlaender,bruoth,chang,ghoshal,novotny,ortiz}.

In this paper we investigate an extremal property of rankings that has not yet
been explored. We consider the maximum number of edges that may be added to
$G$ without changing the rank number. Since the maximum number of edges that
can be added to a graph without changing the rank number varies with each
particular ranking, we will focus on families where an optimal ranking has a
specific structure. Here we investigate the problem for $G=P_{2^{k-1}}$,
$C_{2^{k}}$, $K_{m_{1},m_{2},\dots,m_{t}}$, and the union of two copies of
$K_{n}$ joined by a single edge.

In addition to determining the maximum number of edges that may be added to $G$
without changing the rank number we provide an explicit characterization of which
edges change the rank number when added to $G$, and which edges do not. That is,
given a vertex $v_n$  in $n$th position in the graph, we provide an algorithm to
add a new edge with $v_n$ as one of its vertices to  the graph $G$ without
changing its ranking. For this construction  we use the binary representation of $n$
to determine the position of the second vertex of the new edge. We also construct the
maximum number of edges, so called \emph{good edges}, that can be added to the graph
without changing its ranking. We enumerate the  maximum number of good edges.

\section{Preliminaries}

In this section we review elementary properties and known results about rankings.

A labeling $f:V(G)\rightarrow\{1,2,\dots,k\}$ is a $k$-\textit{ranking} of a
graph $G$ if and only if $f(u)=f(v)$ implies that every $u-v$ path contains a
vertex $w$ such that $f(w)>f(u)$. Following along the lines of the chromatic
number, the \textit{rank number of a graph} $\chi_{r}(G)$ is defined to be the
smallest $k$ such that $G$ has a $k$-ranking. If $H$ is a subgraph of $G$,
then $\chi_{r}(H)\leq\chi_{r}(G)$ (see \cite{ghoshal}).

We use $P_{n}$ to represent the path with $n$ vertices. It is well
known that a ranking of $P_{n}$ with $V\left(  P_{n}\right)  =\left\{
v_{1},v_{2},...,v_{n}\right\}  $ and $\chi_{r}(P_{n})$ labels can be
constructed by labeling $v_{i}$ with $\alpha+1$ where $2^{\alpha}$ is the
largest power of $2$ that divides $i$ \ \cite{bodlaender}. We will call this
ranking the \textit{standard ranking of a path}.

We use $C_{2^{k}}$ to denote a cycle with $2^{k}$ vertices. A multipartite
graph with $t$ components is denoted by $K_{m_{1},m_{2},\dots,m_{t}}$ where
the $i${th} component has $m_{i}$ vertices. The complete graph with $n$
vertices is denoted by $K_{n}$.

Let $\Gamma$ and $H$ be graphs with $V(H)\subseteq V(\Gamma)$ and $E(H)\cap
E(\Gamma)=\emptyset$. We say that an edge $e\in H$ is \emph{good} for $\Gamma$
if $\chi_{r}\left(  \Gamma\cup\{e\}\right)  =\chi_{r}(\Gamma)$, and $e$ is
\emph{forbidden} for $\Gamma$ if $\chi_{r}(\Gamma\cup\{e\})>\chi_{r}(\Gamma)$. We use
$\mu (G)$ to represent the cardinality of the maximum set of good edges for $G$.

For example, in Figure \ref{figure1} we show a ranking of a graph $P_{2^{4}-1}\cup H_{P}$
where $H_{P}$ is the set of all good edges for $P_{2^{4}-1}$. The set of
vertices of $P_{2^{4}-1}$ is $\{v_{1},\dots,v_{15}\}$. We can see that $\chi_{r}(P_{2^{4}-1}\cup H_{P})=\chi_{r}(P_{2^{4}-1})=4$
and that $E(H_{P})$ is comprised of 20 good edges. That is, $\mu (P_{2^{4}-1})=20$.
Theorems \ref{maintheorem} and \ref{maintheoremcomponents} give  necessary and sufficient conditions to determine
whether graphs $G=P_{2^{k}-1}\cup H_{P}$ and $P_{2^{k}-1}$ have the same rank number.

\begin{figure} [htbp]
\begin{center}
\psfrag{v1}[c]{$v_1$} \psfrag{v15}[c]{$v_{15}$} \psfrag{1}[c]{$1$}
\psfrag{2}[c]{$2$} \psfrag{3}[c]{$3$} \psfrag{4}[c]{$4$}
\includegraphics[width=110mm]{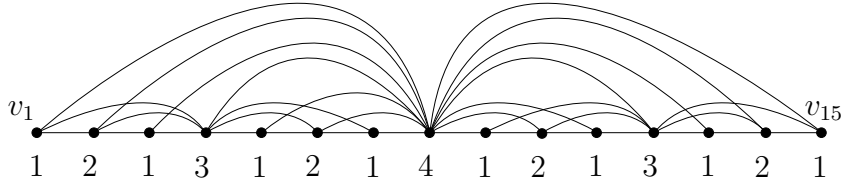}
\end{center}
\caption{$P_{2^4-1} \cup H_P$} \label{figure1}
\end{figure}

Figure \ref{figure2} Part (a) shows the graph $G=C_{2^{4} }\cup H$
where $H$ is the set of all good edges for $C_{2^{4}}$.  We can see that
 $\chi_{r}(C_{2^{4}}\cup H)=\chi_{r}(C_{2^{4}})=5$ and that $E(H)$ is comprised
 of 33 good edges. That is, $\mu (C_{2^{4}})=33$. Theorem \ref{goodarcscyclecorollay} gives
necessary and sufficient conditions to determine whether the graphs
$G=C_{2^{k}}\cup H$ and $C_{2^{k}}$ have the same rank number.

\begin{figure} [htbp]
\begin{center}
\psfrag{v1}[c]{$v_1$} \psfrag{v15}[c]{$v_{15}$}
\psfrag{v16}[c]{$v_{16}$} \psfrag{1}[c]{$1$} \psfrag{2}[c]{$2$}
\psfrag{3}[c]{$3$} \psfrag{4}[c]{$4$} \psfrag{5}[c]{$5$}
\includegraphics[width=40mm]{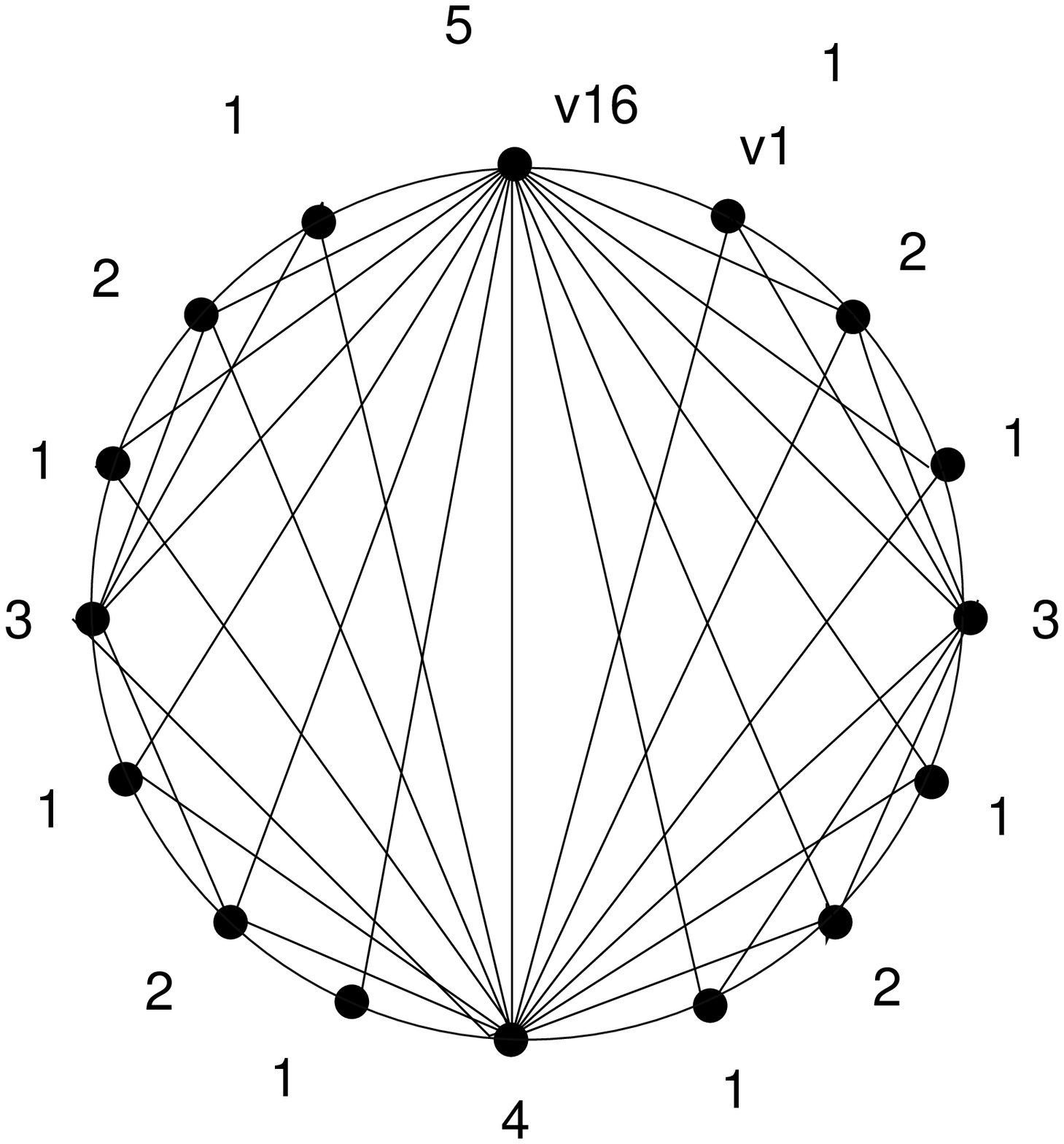} \hspace{2cm}
\includegraphics[width=40mm]{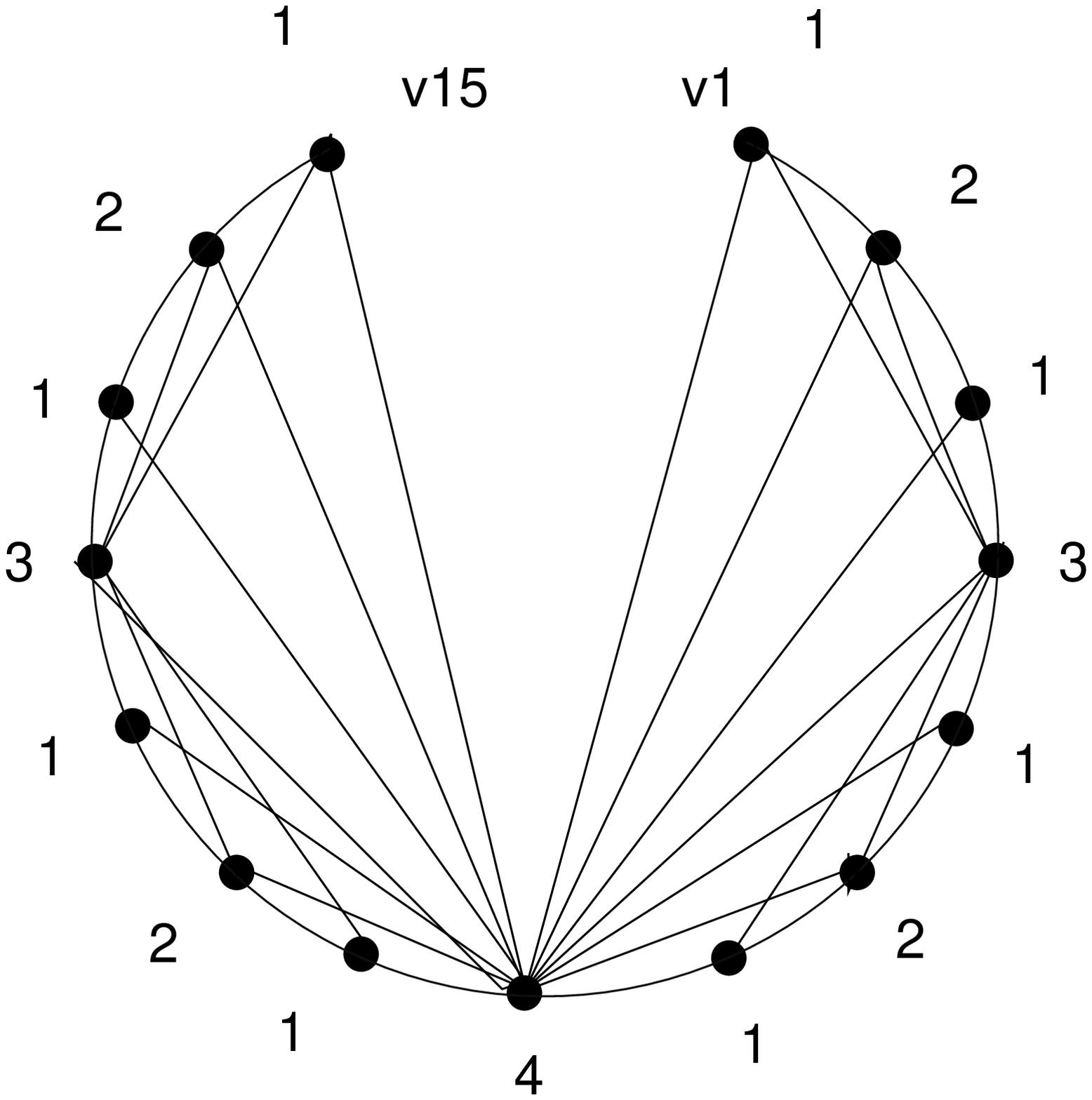}
\end{center}
\caption{(a)\ \ $G:=C_{2^k}\cup H$ \hspace{2cm} (b)\ \
$G':=\left(C_{2^4} \setminus \{ v_{16}\}\right)\cup H' $}
\label{figure2}
\end{figure}

\begin{lemma}
[\cite{bodlaender, bruoth}]\label{rankingofapath} If $k\geq1$, then

\begin{enumerate}
\item  $P_{2^{k}-1}$ has a unique $k-$ranking and $\chi_{r}(P_{2^{k}-1}) =k$.

\item $C_{2^{k}}$ has a unique $k-$ranking and $\chi_{r}(C_{2^{k}}) =k+1$.
\end{enumerate}
\end{lemma}

\section{Enumeration of the Set of Good Edges for $P_{2^{k}-1}$}

In this section we give two ways to find the maximum set of edges that may be added to $G$
without changing the rank number. We give an algorithm to construct a good edge for $G$.
The algorithm is based on the binary representation of $n$, the position of the vertex $v_n$.
That is,  given a vertex $v_n \in G$  in $n$th position,  the algorithm add a new edge,
with $v_n$ as one of its vertices, to  the graph  $G$ without  changing its ranking.
We show that if the graph $G$  is the union of  $P_{2^{t}-1}$ and one edge of the form as indicated
in Procedure 1, then the  ranking of $G$ is equal to the ranking of $P_{2^{t}-1}$.  This guarantees that
the edges constructed using   Procedure 1, are good  edges. We also give sufficient and necessary
conditions to  determine whether a set of edges $H$ is a  set of good set for $P_{2^{t}-1}$.

Since one of our aims is to enumerate the  maximum number of edges that can be added to a graph without
changing its rank, we give a recursive construction of the maximum set of ``good edges". The recursive
construction gives  us a way to count the the number of edges in the set of good edges.

We recall that $(\alpha_{r} \alpha_{r-1} \ldots \alpha_{1}\alpha_{0})_{2}$ with
$\alpha_{h} = 0 \text{ or } 1$ for $0 \le h \le r$ is the binary representation of a positive
integer $b$ if $b=\alpha_{r}2^{r}+\alpha_{r-1}2^{r-1} + \ldots+\alpha_{1}
2^{1}+\alpha_{0}2^{0}$.   We define
\[
g(\alpha_{i})= \left\{
\begin{array}
[c]{ll}
0 & \mbox{if $\alpha_i=1$}\\
1 & \mbox{if $\alpha_i=0$}.
\end{array}
\right.
\]

{\bf Procedure 1.}
Let $V(P_{2^{k}-1})=\{v_{1},v_{2},\ldots,v_{2^{k}-1}\}$ be the set of vertices
of $P_{2^{k}-1}$ and let $H_{P}$ be a graph with $V(H_{P}) = V(P_{2^{k}
-1})$. Suppose that $m<n+1$, $t=\lfloor\log_{2}{m}\rfloor$
and $m=(\alpha_{t} \alpha_{t-1} \ldots  \alpha_{1} \alpha_{0})_{2}$. If $\alpha_j$ is the nonzero
 rightmost entry of $m$, then an edge $e=\{v_{m},v_{n}\}$ is in $H_{P}$ if satisfies any of the following three conditions:

\begin{enumerate}
\item if $m$ is odd then either $n=2^{w}$ for $w>t$ or $n=\Omega(s)$ with
$\Omega(s)=m+1+\sum_{i=1}^{s}g(\alpha_{i})2^{i}$ for $s=1,2,\ldots,t-1$ where
 $m=(\alpha_{t}\alpha_{t-1}\ldots
\alpha_{1} \alpha_{0})_{2}$,

\item $m=2^{j} \cdot(2l+1)$ and $2^{j} \cdot(2l+1)+2 \le n < 2^{j}
\cdot(2l+2)$, for $l>0$,

\item $m=2^{j} \cdot(2l+1)$ and $n=2^{w}$ for $2^{w} \ge2^{j} \cdot(2l+2)$.
\end{enumerate}

{\bf Procedure 2.}
Let $V(C_{2^{k}})=\{v_{1},v_{2},\ldots,v_{2^{k}}\}$ be the set of vertices of
$C_{2^{k}}$. Let $H$ be a graph with $V(H)= V(C_{2^{k}})$. Suppose that $m<n+1$, $t=\lfloor\log_{2}{m}\rfloor$
and $m=(\alpha_{t} \alpha_{t-1} \ldots  \alpha_{1} \alpha_{0})_{2}$. If $\alpha_j$ is the nonzero rightmost entry of  $m$,
then an edge $e=\{v_{m},v_{n}\}$ is in $H$ if satisfies any of the following four conditions:

\begin{enumerate}
\item if $m$ is odd then either $n=2^{w}$ for $w>t$ or $n=\Omega(s)$ with
$\Omega(s)=m+1+\sum_{i=1}^{s}g(\alpha_{i})2^{i}$ for $s=1,2,\ldots,t-1$,

\item $m=2^{j} \cdot(2l+1)$ and $2^{j} \cdot(2l+1)+2 \le n < 2^{j}
\cdot(2l+2)$,

\item $m=2^{j} \cdot(2l+1)$ and $n=2^{w}$ for $2^{w} \ge2^{j} \cdot(2l+2)$
where $l\ge0$,

\item  $1<m < 2^{k}-1$ and $n=2^{k}$.
\end{enumerate}

\begin{lemma}
\label{krankingfunction} Suppose that $f$ is the k-ranking of
$P_{2^{k}-1}$, and $m=(\alpha_{r}\alpha_{r-1}\ldots \alpha_{1}\alpha_{0})_{2}$.
Let $t=\lfloor\log_{2}{m}\rfloor$.

\begin{enumerate}
\item If $\alpha_{i}=0$ for $i<j$ and $\alpha_{j} =1$ , then $f(v_{m})=j$.

\item $f(v_{j})<f(v_{\Omega(i)})$ for $m<j<\Omega(i)$ for $i=1,2, \ldots, t-1$.
\end{enumerate}
\end{lemma}

\begin{proposition}\label{agoodedge}
\label{goodandforbidden} Let $e \not \in P_{2^{t}-1}$ be an edge with vertices
$v_{m}$ and $v_{n}$ where $m<n$. If $e$ is good for $P_{2^{t}-1}$ then $e \in
H_{P}$.
\end{proposition}

\begin{proof}
We proceed with a proof by contradiction assuming that $e\not \in H_{P}$.
Hence we have one of the following cases:

\begin{enumerate}
\item $m$ is odd and $2^{t+1}<n\not =2^{w}$ where $t=\lfloor\log_{2}{m}
\rfloor$,

\item $m$ is odd, $2^{t}<m+1<n<2^{t+1}$ and $n\not =m+1+\sum_{i=1}^{s}
g(\alpha_{i})2^{i}$ for $s=1,2,\ldots,t-1$ with $t=\lfloor\log_{2}{m}\rfloor$
and $m=(\alpha_{t} \alpha_{t-1} \ldots \alpha_{1} \alpha_{0})_{2}$,

\item $m=2^{j} \cdot(2l+1)$ and $2^{t} < n < 2^{t+1}$ where $2^{t} \ge2^{j}
\cdot(2l+2)$ with $l\ge0$.
\end{enumerate}

If Case 1 holds, then $e$ connects vertices $v_{m}$ and $v_{n}$ with $n>2^{t}$.
Suppose that $2^{w}<n<2^{w+1}$. The standard ranking $f$ of a path implies
that $f(v_{\beta})\in\{1,2,\ldots,w\}$ if $\beta<2^{w}$ and that $f(v_{\delta
})\in\{1,2,\ldots,w\}$ for any $2^{w}<\delta<2^{w+1}$. Therefore, there are
two vertices $v_{\gamma}$ and $v_{\rho}$ such that $f(v_{\gamma})=f(v_{\rho
})=w$ with $\gamma<2^{w}<\rho<2^{w+1}$. The path containing the edge $e$ and
vertices $v_{\gamma}$, $v_{m}$, $v_{n}$, and $v_{\rho}$ has two equal labels
with no larger label in between, which contradicts the ranking property. Hence
$m$ is odd and $n=2^{w}$ for $w>t$, and $e$ is good for $P_{2^{k}-1}$.

If Case 2 holds, then $e$ connects vertices $v_{m}$ and $v_{n}$ with
$n<2^{t-1}$. If $m=2^{t+1}-1$, the argument is similar to the above case, so
we suppose that $m\not =2^{t+1}-1$. If $n$ is odd then $f(v_{m})=f(v_{n})=1$,
which is a contradiction.

For the remaining part of this case, we suppose that $m+2<n<2^{t+1}-1$ is even. This implies that
$m \not= 2^{t+1}-1$ and $m\not=2^{t+1}-3$ (note that the Proposition \ref{agoodedge} is now proved for $n<8$).
Therefore, there is at least one nonzero element in
$A=\{\alpha_{2}, \ldots ,\alpha_{t-1}, \alpha_{t} \}$ and let  $i$ be the smallest subscript such that
$\alpha_{i}\in A$ and  $\alpha_{i}=0$. This give rise to two subcases for the
location of \emph{n}:
\begin{enumerate}
\item[(a)] $n<\omega$ with $\omega=\left(  m+1+g(\alpha_{i})2^{i}\right)  $.

\item[(b)] $\Omega(r)<n<2^{t+1}$ where $r$ is the largest number for which the
inequality holds.
\end{enumerate}

If subcase (a) holds, then $\alpha_{j}=1$ for $j<i$. This implies that first
number equal to one in the binary notation of $m+1$ is in position $i+1$. This
and Lemma \ref{krankingfunction} imply that $f(v_{m+1})=i+1$. Therefore,
$f(v_{\omega})=i+2$,  since $m+1<n<\omega$. The definition of the ranking
function $f$, implies that $f(v_{\beta})\in\{1,2,\ldots,i\}$ if $m+1<\beta
<\omega$ and that $f(v_{\delta})\in\{1,2,\ldots,i\}$ for any $\delta<m+1$.
Therefore, there are two vertices $v_{\gamma}$ and $v_{\rho}$ such that
$f(v_{\gamma})=f(v_{\rho})=i$ with $\gamma<m+1<\rho<\omega$. The path
containing the edge $e$ and vertices $v_{\gamma}$, $v_{m}$, $v_{n}$, and
$v_{\rho}$ has two equal labels with no bigger label in between, which is a contradiction.

Suppose that subcase (b) holds. From Lemma \ref{krankingfunction} we know that
$f(v_{d})<f(v_{\Omega(r)})$ for $m<d<r$, in particular we deduce that
$f(v_{\Omega(i)})<f(v_{\Omega(r)})$ if $i<r$. Let $w=f(v_{\Omega(r)})$. This, $P_{2^{k}-1}$
and the definition of $f$ imply that $f(v_{\beta})\in \{1,2,\ldots,w-1\}$ for any $\beta<\Omega(r)$
and that $f(v_{\delta} )\in\{1,2,\ldots,w-1\}$ for any $\Omega(r)<\delta<\Omega(r+1)$. This implies
that there are two vertices $v_{\gamma}$ and $v_{\rho}$ such that
$f(v_{\gamma})=f(v_{\rho})=w-1$ with $\gamma<\Omega(r)<\rho<\Omega(r+1)$. The
path containing the edge $e$ and vertices $v_{\gamma}$, $v_{m}$, $v_{n}$, and
$v_{\rho}$ has two equal labels with no larger label in between, which is a contradiction.

Finally suppose that Case 3 holds. That is, we suppose that every edge $e$
connecting the vertex $v_{m}$ and $v_{n}$ is good, with $m=2^{j}\cdot(2l+1)$
and $2^{t}<n<2^{t+1}$ where $2^{t}\geq2^{j}\cdot(2l+2)$ with $l\geq0$. This
implies that for any $s\leq2^{t}$, the label $f(v_{s})\in\{1,2,\ldots,t\}$, in
particular $f(v_{m})=j+1<t$. Since $2^{t}<n<2^{t+1}$, the coloring
$f(v_{n})\in\{1,2,\ldots,t\}$. Then there are vertices $v_{s}$ and
$v_{s^{\prime}}$ with $f(v_{s})=f(v_{s^{\prime}})=t$ for $s<2^{t}$ and
$2^{t}<s^{\prime}< 2^{t+1}$. This is a contradiction since the path containing the
edge $e$ and connecting vertices $v_{s}$, $v_{m}$, $v_{n}$ and $v_{s^{\prime}
}$, does not have a label larger than $t$.
\end{proof}

\begin{theorem} \label{maintheorem} Let $G=P_{2^{k}-1}\cup H_{P}$.
The set $E(H_{P})$ is the set of good edges for $P_{2^{k}-1}$.
\end{theorem}

Theorem \ref{maintheorem} can be proved using Proposition \ref{agoodedge}, so we
omit the proof. This Theorem is equivalent to Theorem  \ref{maintheoremcomponents} Part 1.
The proof of Theorem \ref{maintheoremcomponents} counts the maximum number of good edges.
We now give some definitions that are going to be used in Lemma \ref{components} and
Theorem \ref{maintheoremcomponents}. Let $G$ be a graph with $f$ as its $k$-ranking.
We define $A_j= \{v \in V(G)\mid f(v)\ge j\}$ and use $\mathcal{C} (A_j)$ to denote the set of
all component of $G \setminus A_j$. If $\mathcal{C} \in \mathcal{C} (A_j)$ and $v \in V( \mathcal{C})$,  then
\[E(v)=\{vw\mid w\in V(\mathcal{C}) \text{ and $w$ not adjacent to $v$} \}.\]
We denote by $E_j$ the union of all sets of the form $E(v)$ where $f(v)=j-1$, the vertex with maximum label in the component. The union is over all components in $\mathcal{C} (A_j) $. That is,

\[E_j=\bigcup_{\begin{tabular}{c}
                    $v \in \mathcal{C}$, $f(v)=j-1$ \\
                    $\mathcal{C}\in \mathcal{C} (A_j) $ \\
                \end{tabular}
} E(v). \]

\begin{lemma} \label{components} If $3< j\le n$, then
\begin{enumerate}
\item $P_{2^n-1} \setminus A_j$ has $2^{n-j}$ components of the form  $P_{2^{j-1}-1}$.

\item If $\mathcal{C}$ is a component of $P_{2^n-1} \setminus A_j$ and $f(v)=j-1$ for some $v \in \mathcal{C}$,
then  $E(v)$ is a set of good edges for $\mathcal{C}$.

\item $\chi_{r}(P_{2^n-1}\cup E_j)=\chi_{r}(P_{2^n-1}).$
\end{enumerate}

\end{lemma}

\begin{proof} For this proof we denote by $f$ the $k$-ranking of $P_{2^n-1}$. We prove Part 1.  Let $u_1$ and $u_2$
be vertices in $P_{2^n-1}$ with $f(u_1)\ge j$ and $f(u_2)\ge j$ and if $w$ is a vertex between $u_1$ and $u_2$
then $f(w)<j$. Since $P_{2^n-1}$ has unique optimal ranking, every $2^{j-1}$ vertices there is a vertex with label greater
than or equal to $j$ (counting from leftmost vertex). This implies that there is a path, of the form  $P_{2^{j-1}-1}$,
connecting all vertices  between $u_1$ and $u_2$, not including $u_1$ and $u_2$. This proves that all components
of $P_{2^n-1}\setminus A_j$ are of the form $P_{2^{j-1}-1}$, and the total number of those components
is $\lceil (2^{n}-1) / 2^{j-1}\rceil = 2^{n-j+1}$.

Proof of Part 2. Let $v$ be a vertex in $\mathcal{C}$ with  $f(v)=j-1$. Since $j-1$, is the largest label in $\mathcal{C}$),
it easy to see that every path containing edges of $E(v)$ does not contribute to increase the ranking of the $\mathcal{C}$.

Proof of Part 3.  Let $G=P_{2^n-1}\cup E_j$ for some  $3< j \le n$. Let $v_1$ and $v_2$ be vertices in $G$ with
$f(v_1)=f(v_2)$, suppose that both vertices are connected by a path $P$. Suppose $v_1$ and $v_2$ are in the
same component  $\mathcal{C}\in \mathcal{C} (A_j)$, then $f(v_1)=f(v_2)<j$. If $P$ is a path of $\mathcal{C}$,
then by the heredity property from $P_{2^n-1}$, there is a vertex in $P$ with a label larger than $f(v_1)$.  We now
suppose that $P$ is not a path of $\mathcal{C}$. Let $v$ be the vertex in  $\mathcal{C}$ with $f(v)=j-1$. These
two last facts and the definition of $E(v)$ imply that $P$ contains an edge in $E(v)$. Thus, there is a vertex in
$P$ with a label larger than $f(v_1)$.

We suppose $v_1$ and $v_2$ are in different components  of $P_{2^n-1}\setminus A_j$. So, $f(v_1)=f(v_2)<j$.
By the definition of $G$ and $E_j$ we see that any path in $G$ connecting two vertices in different component of
$P_{2^n-1}\setminus A_j$ must have at least one vertex in $A_j$. Since vertices in $A_j$ have labels larger $j-1$,
there is a vertex in $P$ with a label larger than $f(v_1)$.

We now suppose $v_1$ and $v_2$ are in $A_j$. So,  $f(v_1)=f(v_2)\ge j$. By definition of $k$-ranking there is vertex $w$
in a subpath of $P_{2^n-1}$ that connects those two vertices, with $f(w)>f(v_1)$. Note that $w \in A_j$. Since $w$ does
not belong to any of the components in $\mathcal{C} (A_j)$, any other path connecting those two vertices must contain
$w$. Therefore, $w\in P$. This proves Part 3.
\end{proof}

\begin{theorem} \label{maintheoremcomponents} If $n>3$, then
\begin{enumerate}
\item $\chi_{r}(P_{2^{n}-1} \cup \bigcup_{j=4}^{n} E_j)= \chi_{r}(P_{2^{n}-1} )=n$ if and only if $\bigcup_{j=4}^{n} E_j$ is the set of good edges for $P_{2^{n}-1}$.

\item $\bigcup_{j=4}^{n} E_j = E(H_{P}).$

\item $\mu (P_{2^n-1})= (n-3)2^n+4$.

\end{enumerate}
\end{theorem}

\begin{proof} For this proof we denote by $f$ the standard $k$-ranking of $P_{2^n-1}$.
We prove of Part 1.  The  proof that the condition is sufficient is straightforward.

To prove that the condition is necessary we use induction. Let $S(t)$ be the statement
\[\chi \left(P_{2^n-1}\cup \bigcup_{j=4}^{t} E_j\right)=\chi (P_{2^n-1}).  \]
Lemma \ref{components} Part 3. proves $S(4)$. Suppose the $S(k)$ is true for some $ 4\le k <n$.
Let
\[ G_0= P_{2^n-1}\cup \bigcup_{j=4}^{k} E_j \text{ \; and \; } G_1= P_{2^n-1}\cup \bigcup_{j=4}^{k+1} E_j.\]
Let $v_1$ and $v_2$ be vertices in $G_1$ with $f(v_1)=f(v_2)$, and suppose that both vertices are
connected by a path $P$. Suppose that $v_1$ and $v_2$ are in the same component
$\mathcal{C} \in \mathcal{C}(A_{k+1})$, then $f(v_1)=f(v_2)<j$. If $P$ is a path of $\mathcal{C}$,
by the heredity property from $G_0$, there is a vertex in $P$ with a label larger than $f(v_1)$.

We suppose that $P$ is not a path of $\mathcal{C}$. Let $v$ the vertex in  $\mathcal{C}$
with $f(v)=k$. These two last facts and definition of $E(v)$ imply that $P$ contains an edge
in $E(v)$. Therefore, there is a vertex in $P$ with a label larger than $f(v_1)$.

We suppose $v_1$ and $v_2$ are in different components  of $G_1\setminus A_{k+1}$. So,
$f(v_1)=f(v_2)<k+1$. By definition of $G_1$ and $E_{k+1}$ we see that any path in $G_1$
connecting two vertices in different component of $G_1\setminus A_{k+1}$ has at
least one vertex in $A_{k+1}$. Since vertices in $A_{k+1}$ have labels larger than $k$,
there is a vertex in $P$ with a label larger than $f(v_1)$.

We now suppose that $v_1$ and $v_2$ are in $A_{k+1}$. So,  $f(v_1)=f(v_2)\ge k+1$. By definition
of ranking there is a vertex $w$ in a path of $G_1$, that connect those two vertices, with
$f(w)>f(v_1)$. Note that $w \in A_{k+1}$. Since $w$ does not belong to any of the
components of $G_1\setminus A_j$, any other path connecting those two vertices must
contain $w$. Therefore, $w\in P$.  This proves that $S(k+1)$ is true. Thus,
$\bigcup_{j=3}^{n} E_j$ is a set of good edges for $P_{2^{n}-1}$.

We now prove that $\bigcup_{i=3}^{n} E_i$ is the largest possible set of good edges
for $P_{2^{n}-1}$. Suppose that $uv$ is a good edge for $P_{2^{n}-1}$ with $f(v)<f(u)=j$.
If the vertices $u$ and $v$ are in the same component of $P_{2^{n}-1} \setminus A_{j+1}$,
then is easy to see that $uv \in E_{j+1}$. Note that $j$ is the largest label in each
component of $P_{2^{n}-1} \setminus A_{j+1}$. If $u$ and $v$ are in different component
of $P_{2^{n}-1} \setminus A_{j+1}$, then $uv$ give rise to a path $P$ connecting $u$ and
a vertex  $w$ where $w$ and $v$ are in the same component and $f(w)=j$. That is a
contradiction, because $f(u)=f(w)=j$ and $P$ does not have label larger than $j$.
This proves that $\bigcup_{j=3}^{n} E_j$ is the set of good edges of $P_{2^{n}-1} $.

The prove of Part 2. is straightforward from Theorem \ref{maintheorem} and Part 1.

Proof of  Part 3. It easy to see that the vertex $v_{2^{n-1}}$ of $P_{2^n-1}$ has label $n$.
That is, $v_{2^{n-1}}$ is the vertex with largest color in $P_{2^n-1}$. Therefore,
$P_{2^n-1}\setminus A_{n}$ has  exactly two components. Note that each component is equal to
$P_{2^{n-1}-1}$ and that  the cardinality of $E(v_{2^{n-1}})$ is $(2^n-1)-3$. Since $v_{2^{n-1}}$ is
the vertex with the largest color in $P_{2^n-1}$, it is easy to see
(from proof of Part 1 and the proofs of Lemma \ref{components}) that the maximum number of edges
that can be added to $P_{2^n-1}$ without changing  the rank is equal to the maximum number
of edges that can be added to each component, $P_{2^{n-1}-1}$, plus all edges in $E(v_{2^{n-1}})$.

Let $a_n=\mu (P_{2^{n}-1})$. Then from the previous analysis we have that
\[a_n= 2 \mu(P_{2^{n-1}-1})+ \mid E(v_{2^{n-1}}) \mid.\]
This give rise to the recurrence relation
$a_n= 2 a_{n-1}+ 2^{n}-4$. Therefore, solving the recurrence relation we have that
$a_n= (n-3)2^n+4$. This proves Part 3.
\end{proof}

\section{Enumeration of the Set of Good Edges for $C_{2^{k}}$}

In this section we  use the results in the previous section to find the maximum set of edges
that may be added to $C_{2^{k}}$  without changing the rank number (good edges).

Suppose that $\Gamma$ represents any of the following graphs;
$C_{2^{k}}$, $K_{m_{1},m_{2},\dots,m_{t}}$ or the graph defined by the
union of two copies of $K_{n}$ joined by an edge $e$. In this section we give
sufficient and necessary conditions to determine whether a set of edges $H$ is
a good set for $\Gamma$. For all graphs in this section we count the number of
elements in each maximum set of good edges.

We recall that in Figure 2 Part (a) we show the graph $G=C_{2^{4}}\cup H$ where $H$ is the
set of all good edges for $G$. So, $\chi_{r}(G)=\chi_{r}(C_{2^{4}})=5$. In
Figure 2 Part (b) we show the graph $G^{\prime}=(C_{2^{k}}\setminus
\{v_{16}\})\cup H^{\prime}$ where $H^{\prime}$ is the set of all good edges
for $G^{\prime}$. Since the graph in Figure 2 Part (b) is equivalent to the
graph in Figure 1, Theorem \ref{maintheoremcomponents} can be applied to this graph.
Theorem \ref{goodarcscyclecorollay} gives sufficient and necessary conditions
to determine whether the graphs $G=C_{2^{k}}\cup H$ and $C_{2^{k}}$ have the
same rank number and counts the maximum number of good edges.

\begin{proposition}
\label{goodarcscycle} If an edge $e$ is good for $P_{2^{k}-1}$ then $e$ is
good for $C_{2^{k}}$.
\end{proposition}

\begin{proof}
Since the standard ranking of $P_{2^{k}-1}$ is contained in the ranking of the
cycle $C_{2^{k}}$, and the additional vertex is given the highest label, it
follows that if edges are good for the path, they will be good for the cycle.
\end{proof}

Let $V:=\{v_{1},v_{2},\dots,v_{2^{k}}\}$ be the set of vertices of $C_{2^{k}}
$. Notice that set of edges of $P_{2^{k}-1}$ is $V\setminus\{v_{2^{k}} \}$. We
define
\[
H_{C} = H_{P} \cup\{e \mid e \not \in C_{2^{k}} \text{ and vertices }
\{v_{2^{k}}, v_{i}\}, \text{ with } i \in\{2, \dots, 2^{k}-2 \} \}.
\]

\begin{theorem} \label{goodarcscyclecorollay}  If $k>3$, then

\begin{enumerate}
\item $\chi_{r}(C_{2^{k}} \cup H_{C})= \chi_{r}(C_{2^{k}})=k+1$ if and only if $H_{C}$ is the set of good edges for $C_{2^{k}}$.

\item  $\mu (C_{2^n})= (n-2)2^n+1$.

\end{enumerate}
\end{theorem}

\begin{proof} To prove Part 1, we first show the condition is sufficient. Suppose
$\chi_{r}(C_{2^{k}} \cup H_{C} )=\chi _{r}(C_{2^{k}})=k+1$.
Suppose $E(H)$ is not a set of good edges. Thus, $E(H)$
contains a forbidden edge, therefore the rank number of $C_{2^{k}}$ is greater
than $k+1$. That is a contradiction.

Next we show the condition is necessary. It is known that $\chi_{r}(C_{2^{k}})=k+1$
and that this ranking is unique (up to permutation of the two largest
labels) \cite{bruoth}. Let $f$ be a ranking of $C_{2^{k}}$ with $k+1$ labels
where $f(v_{2^{k}})=k+1$.

Let $e_{1}=\{v_{2^{k}-1},v_{2^{k}}\}$ and $e_{2}=\{v_{2^{k}},v_{1}\}$ be two
edges of $C_{2^{k}}$ and let $H^{\prime}$ be the graph formed by edges of $H$
with vertices in $V^{\prime}=V(H)\setminus\{v_{t}\}=\{v_{1},v_{2} ,\ldots,v_{2^{k}-1}\}$.
Theorem \ref{maintheoremcomponents} Parts 1 and 2 imply
 that $E(H^{\prime})$ is a set of good edges for the graph
 $C_{2^{k}}\setminus\{e_{1},e_{2}\}$ if and
only if $\chi_{r}(C_{2^{k}}\setminus\{e_{1},e_{2}\}\cup H^{\prime})=k$ (see
Figure 2 Part (b)). Note that $V^{\prime}$ is the set of vertices of
$C_{2^{k}}\setminus\{e_{1},e_{2}\}$. The vertices of $C_{2^{k}}\setminus
\{e_{1},e_{2}\}\cup H^{\prime}$ have same labels as vertices $V^{\prime}$.
Combining this property with $f(v_{2^{k}})=k+1$ gives $\chi_{r}(C_{2^{k}}\cup
H^{\prime})=k+1$.

We now prove that $\chi_{r}(C_{2^{k}}\cup H)=k+1$. Let $e$ be an edge in
$H\setminus H^{\prime}$. Therefore, the end vertices of $e$ are $v_{2^{k}}$
and $v_{n}$ for some $2\leq n\leq2^{k}-2$. From the ranking $f$ of a cycle we
know that $f(v_{2^{k}})=k+1$ and $f(v_{n})<k+1$. Hence we do not create a new
path in $C_{2^{k}} \cup H_{C}$ connecting vertices with labels $k+1$.

Proof of Part 2.
Let $W$ be the set of edges of the form $\{v_{2^{n}}, v_i\}$ for $i= 2,3, \ldots, 2^{n}-2$.
The cardinality of $W$ is $2^n-3$. From Proposition \ref{goodarcscycle} we know that all
good edges for $P_{2^n-1}$ are also good for $C_{2^n}$. Therefore, the maximum number
of edges that can be added to $C_{2^n}$ without changing the rank is equal to maximum
number of edges that can be added to $P_{2^n-1}$ plus all edges in $W$. Thus,
$\mu (C_{2^n})= \mu (P_{{2^n}-1})+\mid W \mid$. This and Theorem \ref{maintheoremcomponents} Part 3.
imply that  $\mu (C_{2^n})= (n-3)2^n+4 +2^n-3$. Therefore,  $\mu (C_{2^n})= (n-2)2^n+1$.
\end{proof}

\begin{theorem} \label{productmultipartitegraphs} Let $m_1, m_{2}, \ldots, m_t$ be positive integers with $m_{1}=\max\{m_{i}\}_{i=1}^{t}$. If $G=K_{m_{1},m_{2},\dots,m_{t}}$  is a
multipartite graph, then

\begin{enumerate}
\item any edge connecting two vertices in a part of order $m_{1}$ is
forbidden, and

\item any edge connecting any two vertices in any part of order $m_{i}$ where $i\neq1$
is good.

\item  $\mu (K_{m_{1},m_{2},\dots,m_{t}})= \sum_{i=2}^t \dfrac{(m_{i}-1)m_{i}}{2}.$

\end{enumerate}
\end{theorem}

\begin{proof}
Let $W=\{w_{1},w_{2},\ldots,w_{m_{1}}\}$ be the set of vertices of the part of
$G$ with order $m_{1}$. Let $V=\{v_{2},v_{3},\ldots,v_{r}\}$ be the set of
vertices of $G\backslash W$. We consider the function
\[
f(x)=\left\{
\begin{array}
[c]{cc}
1 & \text{ if }x\in W\\
i & \text{ if }x=v_{i}\text{ for some }v_{i}\in V.
\end{array}
\right.
\]

To see that $f$  is  a minimum ranking of $G$, note that reducing any label
violates the ranking property.

Proof of Part 1. Any edge connecting two vertices in $W$ gives rise to a
path connecting to vertices with same label.

Proof of Part 2. Any edge connecting two vertices in $V$ does not create
any new path with vertices with the same label.

Proof of Part 3. Let $U=\{u_1, \ldots ,u_{m_s} \}$ the set of vertices of the part of
$K_{m_{1},m_{2},\dots,m_{t}}$ with $m_s$ vertices and with $s\not =1$.
Let $E_{m_s}$ be set of edges of the form
$\{v_{i}, v_{j}\}$ for $i, j$ in $\{1, 2, \ldots m_s-2\}$ and $j>i+1$.
From the proof of Theorem \ref{productmultipartitegraphs} we know that
$E_{m_s}$ is a set of good edges of
$K_{m_{1},m_{2},\dots,m_{t}}$ for $s=2, 3, \ldots, t$ (if $s=1$, then $E_{m_1}$ will be a forbidden set).
The cardinality of $E_{m_s}$ is $(m_s-1)m_s/2$ for $s=2, 3, \ldots, t$.
This implies that \[\mu (K_{m_{1},m_{2},\dots,m_{t}})= \sum_{s=2}^t \frac{(m_{s}-1)m_{s}}{2}.\] This proves Part 3.
\end{proof}

Let $G_{5}$ be the graph defined by the union of two copies of $K_{5}$ joined
by an edge $e$. In Figure 3 we show $G_{5}\cup H$ where $H$ is the set of all
good edges for $G_{5}$. So, $\chi_{r}(G_{5}\cup H)$ $=$ $\chi_{r}(G_{5})=6$ and $\mu(G_{5})=8$.
We generalize this example in Theorem \ref{unioncompletegraphs}.

\begin{figure} [htbp]
\begin{center}
\psfrag{a}[c]{$Ale es una gueva$} \psfrag{v15}[c]{$v_{15}$}
\psfrag{v16}[c]{$v_{16}$} \psfrag{1}[c]{$1$} \psfrag{2}[c]{$2$}
\psfrag{3}[c]{$3$} \psfrag{4}[c]{$4$} \psfrag{5}[c]{$5$}
\psfrag{e}[c]{$e$} \psfrag{6}[c]{$6$}
\includegraphics[width=75mm]{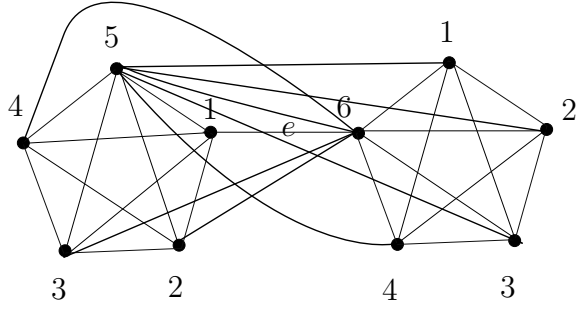}
\end{center}
\caption{$\chi_r(G_5 \cup H)=6$} \label{figure3}
\end{figure}

\begin{theorem}
\label{unioncompletegraphs} Let $G_{n}$ be the union of two copies of $K_{n}$
joined by an edge. Then,

\begin{enumerate}
\item any edge connecting a vertex with highest label in one
part with any other vertex in the other part is good. All other edges are
forbidden. Moreover, if $H$ is the set of all good edges for $G_{n}$, then
$\chi_{r}(G_{n}\cup H)=\chi_{r}(G_{n})=n+1$.

\item $\mu (G_{n})= 2(n-1)$.

\end{enumerate}
\end{theorem}

\begin{proof} To prove Part 1. we suppose that
$G_{n}=K\cup K^{\prime}\cup e$ where $K=K^{\prime}=K_{n}$. Let
$W=\{w_{1},w_{2},\ldots,w_{n}\}$ be the set of vertices of $K$ and let
$V=\{v_{1},v_{2},\ldots,v_{n}\}$ be the set of the vertices of $K^{\prime}$
and $\{w_{1},v_{n}\}$ the set of vertices of $e$. We consider the function
\[
f(x)=\left\{
\begin{array}
[c]{ll}
i & \text{ if }x=w_{i}\text{ for some }w_{i}\in W\\
i & \text{ if }x=v_{i}\text{ for some }v_{i}\in V\setminus\{v_{n}\}\\
n+1 & \text{ if }x=v_{n}.
\end{array}
\right.
\]

\noindent It is easy to see that $f$ is a minimum ranking of $G_{n}$ and that
$\chi_{r}(G_{n})=n+1$. Let
\[
H_{1}=\{e \mid e \not \in G_{n} \text{ is an edge with vertices } w_{n}, v_{i}
\text{ for some } i \in\{1, 2, \ldots, n-1 \} \}
\]
and
\[
H_{2}=\{e \mid e \not \in G_{n} \text{ is an edge with vertices } v_{n}, w_{i}
\text{ for some } i \in\{1, 2, \ldots, n-1 \} \}.
\]
We prove that $H=H_{1}\cup H_{2}$ is the set of good edges for $G_{n}$.

From the definition of $f$ we know the labels of the vertices in $K$ are
distinct and all of the labels in $K^{\prime}$ are distinct. The combination
of these properties and the definition of $f(v_{n})$ implies that if an edge
$e$ connects one vertex in $K$ with $v_{n}$ it does not create a new edge
connecting two edges with same label. Similarly, if an edge $e$ connects one
vertex in $K^{\prime}$ with $w_{n}$ it does not create a path connecting two
edges with the same label. This proves that $H$ is the set of good edges for
$G_{n}$ and that $\chi_{r}(G_{n}\cup H)=\chi_{r}(G_{n})=n+1$.

Suppose that an edge $e$ connects the vertices $w_{i}\in K$ and $v_{j}\in
K^{\prime}$ with $i\leq j\not =n$. The path $w_{j}w_{i}v_{j}$ has two vertices
with same label without a larger label in between. The proof of the case
$j\leq i\not =n$ is similar. Hence if $e\not \in H$, then $e$ is\ a forbidden edge.

Proof of Part 2. From proof of Part 1. we can see that $H_{1}$ and $H_{2}$
are the sets of good edges of $G_{n}$ and that the cardinality of each set is $n-1$.
So, $\mu (G_{n})= \mid H_1 \mid + \mid H_2 \mid= 2(n-1)$.
\end{proof}

\section* {Acknowledgment}
The authors are indebted to the referees for their comments and corrections that helped to improve
the presentation.

\end{document}